\documentclass[12pt]{article}
\usepackage{url}

\newcommand{\half}{0.5}

\newcommand{\ang}[1]{\langle#1\rangle}

\newcommand{\goes}{\rightarrow}

\newcommand{\nat}{{\sf N}}

\newcommand{\xvec}[1]{\ifcase 3{#1} {\ang {x_1,x_2,x_3} } \else 
\ifcase 4{#1} {\ang{x_1,x_2,x_3,x_4}} \else {\ang {x_1,\ldots,x_{#1}}}\fi\fi}
\newcommand{\yvec}[1]{\ifcase 3{#1} {\ang {y_1,y_2,y_3} } \else 
\ifcase 4{#1} {\ang{y_1,y_2,y_3,y_4}} \else {\ang {y_1,\ldots,y_{#1}}}\fi\fi}
\newcommand{\zvec}[1]{\ifcase 3{#1} {\ang {z_1,z_2,z_3} } \else 
\ifcase 4{#1} {\ang{z_1,z_2,z_3,z_4}} \else {\ang {z_1,\ldots,z_{#1}}}\fi\fi}
\newcommand{\vecc}[2]{\ifcase 3{#2} {\ang { {#1}_1,{#1}_2,{#1}_3 } } \else
\ifcase 4{#1} {\ang { {#1}_1,{#1}_2,{#1}_3,{#1}_{4} } }
\else {\ang { {#1}_1,\ldots,{#1}_{#2}}}\fi\fi}
\newcommand{\veccd}[3]{\ifcase 3{#2} {\ang { {#1}_{{#3}1},{#1}_{{#3}2},{#1}_{{#3}3} } } \else
\ifcase 4{#1} {\ang { {#1}_{{#3}1},{#1}_{{#3}2},{#1}_{#3}3},{#1}_{{#3}4} }
\else {\ang { {#1}_{{#3}1},\ldots,{#1}_{{#3}{#2}}}}\fi\fi}
\newcommand{\veccz}[2]{\ifcase 3{#2} {\ang { {#1}_0,{#1}_2,{#1}_3 } } \else
\ifcase 4{#1} {\ang { {#1}_0,{#1}_2,{#1}_3,{#1}_{4} } }
\else {\ang { {#1}_0,\ldots,{#1}_{#2}}}\fi\fi}
\newcommand{\xve}[1]{\ifcase 3{#1} {x_1,x_2,x_3} \else 
\ifcase 4{#1} {x_1,x_2,x_3,x_4} \else {x_1,\ldots,x_{#1}}\fi\fi}
\newcommand{\yve}[1]{\ifcase 3{#1} {y_1,y_2,y_3} \else 
\ifcase 4{#1} {y_1,y_2,y_3,y_4} \else {y_1,\ldots,y_{#1}}\fi\fi}
\newcommand{\zve}[1]{\ifcase 3{#1} {z_1,z_2,z_3} \else 
\ifcase 4{#1} {z_1,z_2,z_3,z_4} \else {z_1,\ldots,z_{#1}}\fi\fi}
\newcommand{\ve}[2]{\ifcase 3#2 {{#1}_1,{#1}_2,{#1}_3} \else
\ifcase 4#2 {{#1}_1,{#1}_2,{#1}_3,{#1}_{4}}
\else {{#1}_1,\ldots,{#1}_{#2}}\fi\fi}
\newcommand{\ved}[3]{\ifcase 3#2 {{#1}_{{#3}1},{#1}_{{#3}2},{#1}_{{#3}3}} \else
\ifcase 4#2 {{#1}_{{#3}1},{#1}_{{#3}2},{#1}_{{#3}3},{#1}_{{#3}4}}
\else {{#1}_{{#3}1},\ldots,{#1}_{{#3}{#2}}}\fi\fi}
\newcommand{\fuve}[3]{
\ifcase 3#2
{{#3}({#1}_1),{#3}({#1}_2,{#3}({#1}_3)} \else
\ifcase 4#2
{{#3}({#1}_1),{#3}({#1}_2),{#3}({#1}_3),{#3}({#1}_4)}
\else
{{#3}({#1}_1),\ldots,{#3}({#1}_{#2})}\fi\fi}
\newcommand{\setmathchar}[1]{\ifmmode#1\else$#1$\fi}
\newcommand{\vlist}[2]{%
	\setmathchar{%
		\compound#2\one{#2}\two
		\ifcompound
			({#1}_1,\ldots,{#1}_{#2})
		\else
			\ifcat N#2
				({#1}_1,\ldots,{#1}_{#2})
			\else
				\ifcase#2
					({#1}_0)\or
					({#1}_1)\or
					({#1}_1,{#1}_2)\or 
					({#1}_1,{#1}_2,{#1}_3)\or
					({#1}_1,{#1}_2,{#1}_3,{#1}_4)\else 
					({#1}_1,\ldots,{#1}_{#2})
				\fi
			\fi
		\fi}}

\newif\ifcompound
\def\compound#1\one#2\two{%
	\def\one{#1}
	\def\two{#2}
	\if\one\two
		\compoundfalse
	\else
		\compoundtrue
	\fi}

\newcommand{\xwe}[1]{\ifcase 3{#1} {x_1\wedge x_2\wedge x_3} \else 
\ifcase 4{#1} {x_1\wedge x_2\wedge x_3\wedge x_4} \else {x_1\wedge \cdots \wedge
x_{#1}}\fi\fi}
\newcommand{\we}[2]{\ifcase 3#2 {\ang { {#1}_1\wedge {#1}_2\wedge {#1}_3 } } \else
\ifcase 4{#1} {\ang { {#1}_1\wedge {#1}_2\wedge {#1}_3\wedge {#1}_{4} } }
\else {\ang { {#1}_1\wedge \cdots\wedge {#1}_{#2}}}\fi\fi}

\newcommand{\st}{\mathrel{:}}

\newcommand{\implies}{\Rightarrow}
\newcommand{\into}{\rightarrow}

\newcommand{\s}[1]{\s_{#1}}

\newcommand{\monus}{\;\raise.5ex\hbox{{${\buildrel
    \ldotp\over{\hbox to 6pt{\hrulefill}}}$}}\;}

\newcounter{savenumi}

\newtheorem{theoremfoo}{Theorem}[section] 
\newenvironment{theorem}{\pagebreak[1]\begin{theoremfoo}}{\end{theoremfoo}}

\newtheorem{lemmafoo}[theoremfoo]{Lemma}
\newenvironment{lemma}{\pagebreak[1]\begin{lemmafoo}}{\end{lemmafoo}}
\newtheorem{conjecturefoo}[theoremfoo]{Conjecture}

\newtheorem{conventionfoo}[theoremfoo]{Convention}
\newenvironment{convention}{\pagebreak[1]\begin{conventionfoo}\rm}{\end{conventionfoo}}

\newtheorem{porismfoo}[theoremfoo]{Porism}

\newtheorem{gamefoo}[theoremfoo]{Game}

\newtheorem{corollaryfoo}[theoremfoo]{Corollary}

\newtheorem{claimfoo}[theoremfoo]{Claim}

\newtheorem{openfoo}[theoremfoo]{Open Problem}

\newtheorem{exercisefoo}{Exercise}

\newcommand{\fig}[1] 
{
 \begin{figure}
 \begin{center}
 \input{#1}
 \end{center}
 \end{figure}
}

\newtheorem{potanafoo}[theoremfoo]{Potential Analogue}

\newtheorem{notefoo}[theoremfoo]{Note}
\newenvironment{note}{\pagebreak[1]\begin{notefoo}\rm}{\end{notefoo}}

\newtheorem{notabenefoo}[theoremfoo]{Nota Bene}

\newtheorem{nttn}[theoremfoo]{Notation}
\newenvironment{notation}{\pagebreak[1]\begin{nttn}\rm}{\end{nttn}}

\newtheorem{empttn}[theoremfoo]{Empirical Note}

\newtheorem{examfoo}[theoremfoo]{Example}

\newtheorem{dfntn}[theoremfoo]{Def}
\newenvironment{definition}{\pagebreak[1]\begin{dfntn}\rm}{\end{dfntn}}

\newtheorem{propositionfoo}[theoremfoo]{Proposition}

\newenvironment{proof}
    {\pagebreak[1]{\narrower\noindent {\bf Proof:\quad\nopagebreak}}}{\QED}

\newcommand{\yyskip}{\penalty-50\vskip 5pt plus 3pt minus 2pt}
\newcommand{\blackslug}{\hbox{\hskip 1pt
        \vrule width 4pt height 8pt depth 1.5pt\hskip 1pt}}
\newcommand{\QED}{{\penalty10000\parindent 0pt\penalty10000
        \hskip 8 pt\nolinebreak\blackslug\hfill\lower 8.5pt\null}
        \par\yyskip\pagebreak[1]}

\newcommand{\BBB}{{\penalty10000\parindent 0pt\penalty10000
        \hskip 8 pt\nolinebreak\hbox{\ }\hfill\lower 8.5pt\null}
        \par\yyskip\pagebreak[1]}

\newtheorem{factfoo}[theoremfoo]{Fact}

\newenvironment{block}{\begin{list}{\hbox{}}{\leftmargin 1em
    \itemindent -1em \topsep 0pt \itemsep 0pt \partopsep 0pt}}{\end{list}}

\newcommand{\Erdos}{Erd\"os }

\newcommand{\Sarkozyns}{S{\'a}rk{\"o}zy}

\newcommand{\Szemeredi}{Szemer{\'e}di }

\begin{document}

\newcommand{\sdf}{\rm SDF }
\newcommand{\sdfns}{\rm SDF}
\newcommand{\sdfmod}[1]{{\rm SDFMOD}({#1})}

\renewcommand{\s}[1]{{\rm sdf}{(#1)}}
\newcommand{\smod}[1]{{\rm sdfmod}{(#1)}}

\newcommand{\sizem}{n^{0.7154}}
\newcommand{\size}{n^{0.7334}}
\newcommand{\newsize}{n^{0.7334\cdots}}
\newcommand{\oldsize}{n^{0.733077\cdots}}
 \newcommand{\onemsize}{n^{0.2666}}
\newcommand{\onemsizep}{n^{0.2666}}
\newcommand{\onemsizem}{n^{0.2666}}
\newcommand{\cbone}{c^{1/0.2666}}
\newcommand{\cbtwo}{c^{3.75}}
\newcommand{\COL}{\chi}

\title{Square-Difference-Free Sets of Size $\Omega(\newsize)$}

\author{
{Richard Beigel}
\thanks{Temple University,
Dept. of Computer and Information Sciences,
1805 N Broad St Fl 3,
Philadelphia, PA 19122.
\texttt{professorb@gmail.com}
}
\\ {\small Temple University}
\and
{William Gasarch}
\thanks{University of Maryland,
Dept. of Computer Science and Institute for Advanced Computer Studies,
	College Park, MD\ \ 20742.
\texttt{gasarch@cs.umd.edu}, Partially supported by NSF grant CCR-01-05413
}
\\ {\small Univ. of MD at College Park}
}

\date{}

\maketitle

\begin{abstract}
A set $A\subseteq \nat$ is {\it square-difference free} (henceforth \sdfns)
if there do not exist $x,y\in A$, $x\ne y$,  such that $|x-y|$ is a square.
Let $\s n$ be the size of the largest \sdf subset of $\{1,\ldots,n\}$.
Ruzsa has shown that
$$\s n = \Omega(n^{\half(1+ \log_{65} 7)}) = \Omega(\oldsize)$$
We improve on the lower bound by showing 
$$\s n = \Omega(n^{\half(1+ \log_{205} 12)})= \Omega(\newsize)$$
As a corollary we obtain a new lower bound on the quadratic van der Waerden numbers.
\end{abstract}

\section{Introduction}

\begin{notation}
$\nat$ is the set $\{1,2,3,\ldots\}$.
If $n\in\nat$, and $n\ge 1$ then $[n]=\{1,\ldots,n\}$.
\end{notation}

Van der Waerden proved the following using purely combinatorial techniques (\cite{VDW}, see also \cite{VDWs} for a purely
combinatorial proof with better bounds, \cite{GRS} for an exposition of both of those proofs,
and \cite{Gowers} for a proof using non-combinatorial techniques that provides the best known bounds).

\begin{theorem}\label{th:vdw}
For all $k,c\in\nat$, there exists $W=W(k,c)\in\nat$ such that,
for any $c$-coloring $COL:[W]\into [c]$, there are
$a,d\in\nat$, $d\ne 0$, such that
$$a, a+d, a+2d,\ldots,a+(k-1)d\in [W]$$
$$COL(a)=COL(a+d)=COL(a+2d)= \cdots= COL(a+(k-1)d).$$
\end{theorem}

\Szemeredi~\cite{density} proved the following theorem, which
implies Van der Waerden's theorem:

\begin{theorem}\label{th:sz}
For all $k\in\nat$, for all $0<\alpha<1$, for almost all $n$,
the following holds:
$$(\forall A\subseteq [n])[|A|\ge \alpha n \implies A\hbox{ has an arithmetic 
sequence of length $k$}].$$
\end{theorem}

We now look at generalizations of Theorems~\ref{th:vdw} and~\ref{th:sz}.

Note that in Theorem~\ref{th:vdw} we have the sequence
$$a, a+d, \ldots, a+(k-1)d.$$
Why the functions $d, 2d, \ldots, (k-1)d$?
Can they be replaced by polynomials?
YES (with one condition):

\begin{theorem}\label{th:polyvdw}
For all $c\in\nat$, for all 
$p_1(x),\ldots,p_k(x)\in Z[x]$ such
that $(\forall i)[p_i(0)=0]$ (that is, they all have constant term 0),
there exists a natural number $W=W(p_1,\ldots,p_k;c)$ such that,
for all $c$-coloring $COL:[W]\into [c]$, there exists $a,d\in \nat$, $d\ne 0$, such that
$$a, a+p_1(d),  a+p_2(d), \ldots, a+p_k(d))\in [W]$$ 
$$COL(a)=COL(a+p_1(d))= COL(a+p_2(d))=\cdots= COL(a+p_k(d)).$$ 
\end{theorem}

\begin{note}
Theorem~\ref{th:polyvdw} was proved for $k=1$ by Furstenberg~\cite{pvdwfurst} and
(independently) \Sarkozyns~\cite{pvdwsar}.
Bergelson and Leibman~\cite{pvdw} proved the general result
using ergodic methods. Walters~\cite{pvdww} proved
it using purely combinatorial techniques.
\end{note}

Theorem~\ref{th:sz} also has a polynomial analog:

\begin{theorem}\label{th:polysz}
For all $0<\alpha<1$, 
for all $p_1(x),\ldots,p_k(x)\in Z[x]$ such that $(\forall i)[p_i(0)=0]$,
for almost all $n$,
the following holds:
$$(\forall A\subseteq [n])[|A|\ge \alpha n \implies (\exists a,d\in \nat)
[a,a+p_1(d),a+p_2(d),\ldots,a+p_k(d)\in A]]$$
\end{theorem}

Bergelson and Leibman~\cite{pvdw} proved Theorem~\ref{th:polysz}.
This is how they obtained Theorem~\ref{th:polyvdw}.

\begin{definition}
A set $A\subseteq \nat$ is {\it square-difference free} (henceforth \sdfns)
if there do not exist $x,y\in A$, $x\ne y$,  such that $|x-y|$ is a square.
\end{definition}

\begin{definition}
Let $\s n $ be the size of the largest \sdf subset of $[n]$.
\end{definition}

Theorem~\ref{th:polysz} implies that, for any $0<\alpha<1$, for  almost all $n$,
$$\s n  \le \alpha n.$$
The following bounds are known on $\s n$.
\begin{itemize}
\item
\Sarkozyns~\cite{pvdwsar} proved
$$\s n \le O\bigg (\frac{n(\log\log n)^{2/3}}{(\log n)^{1/3}}\bigg ).$$
\item
Pintz, Steiger, and \Szemeredi \cite{PSS} proved
$$\s n \le O\bigg (\frac{n}{(\log n)^{c_n}}\bigg )$$
where $c_n\goes\infty$. (See also~\cite{wolfsq} for an exposition.)
\item
\Sarkozyns~\cite{pvdwsar2} showed that, for all $\epsilon<0.5$,
$\s n  \ge n^{0.5+\epsilon f(n)}$, where $f(n)=\frac{\log \log\log n}{\log\log n}$.
\item
Ruzsa~\cite{ruzsasq} proved $\s n \ge \Omega(n^{\log_{65} 7}) \ge \Omega(\oldsize)$.
\end{itemize}

We improve on Ruzsa's result by showing
$$\s n\ge \Omega(n^{\log_{205} 12}) =  \Omega(\newsize).$$
Our proof is similar to Ruzsa's; however,
we include proofs for completeness.

\section{An \sdf set of size $\ge \Omega(n^{0.5})$}

We present the result $\s n\ge n^{0.5}$, since it is easy
and, while known~\cite{pvdwsar2}, is not online and seems hard to find.
We do not need this result; however, it is very nice.

Recall Bertrand's Postulate\footnote{Bertrand's Postulate was
actually proven by Chebyshev's.
Bertrand conjectured
that, for all $n>3$, there is a prime between $n$ and $2n-2$.
Bertrand proved it for all $n<3\times 10^6$.
Cheshire proved it completely in 1850. It is usually stated as
we do below. A proof due to \Erdos can be found either in~\cite{HW}
or on Wikipedia.}
which we state as a lemma.

\begin{lemma}
For all $n$ there is a prime $p$ such that $n\le p\le 2n$.
\end{lemma}

\begin{theorem}\label{th:sqrt}
$$\s n = \Omega(n^{0.5}).$$
\end{theorem}

\begin{proof}
By Bertrand's Postulate
there exists a prime $p$ such that
$$
\frac{n^{0.5}}{2} \le p \le n^{0.5}.
$$

Let

$$A= \{ p, 2p, 3p,\ldots, p^2 \}.$$

Clearly, $|A|= p \ge \Omega(\sqrt n )$.
We show that $A$ is \sdfns.

Let $ip$ and $jp$ be two elements of $A$. Note that
$$jp-ip=(j-i)p.$$
We can assume that $i<j$, so
$$1\le i<j\le p.$$
Thus we have $j-i<p$. Hence $(j-i)p$ has only one
factor of $p$, so $jp-ip$ cannot be a square.
\end{proof}

\section{An \sdf set of size $\ge \Omega(\newsize)$}

To obtain large \sdf sets, we will first work with \sdf sets with respect to various moduli.

\begin{convention}\label{co:mod}
Throughout this section when we deal with mod $m$
we will use the set $[m]=\{1,\ldots,m\}$
rather than the more traditional $\{0,\ldots,m-1\}$.
In calculations we may use 0 instead of $m$ for clarity.
For example, if we have that $b_1\equiv b_2 \pmod m $ then
we will feel free to write $b_1-b_2\equiv 0 \pmod m$.
\end{convention}

\begin{definition}
Let $n\in\nat$.
A set $A\subseteq [n]$ is {\it square-difference free mod $n$} (henceforth
$\sdfmod n$)
if there do not exist $x,y\in A$, $x\ne y$, such that $x-y$ is a square mod $n$.
\end{definition}

\begin{definition}
Let $\smod n$ be the size of the largest $\sdfmod n$ set.
\end{definition}

Note that $\smod n \le \s n$. We will obtain lower bounds
for $\s n $ by obtaining lower bounds for $\smod n$.
The next lemma shows how to construct such sets.

\begin{lemma}\label{le:mk}
Assume $m$ is squarefree, $k\ge1$, and
$B,X,S,Y$ are sets such that the following hold:
\begin{enumerate}
\item
$S$ is an $\sdfmod m$ subset of $[m]$.
\item
$X$ is an $\sdfmod {m^{2k-2}}$ subset of $[m^{2k-2}]$.
\item
$B = \{mz + b \mid  z \in \{0,\ldots,m-1\} \wedge b\in S\}$.
Note that $B\subseteq [m^2]$ and 
$|B|=m|S|$.
\item
$Y = \{m^2 x + s \pmod {m^{2k}} \mid  x \in X \wedge s \in B\}$. 
Since $Y$ is defined $\pmod {m^{2k}}$, 
when we use Convention~\ref{co:mod},
we have $Y\subseteq [m^{2k}]$.
Note that $|Y|=|X||B|=m|S||X|$.
\end{enumerate}
Then  $Y$ is an $\sdfmod{m^{2k}}$ subset of $[m^{2k}]$.
\end{lemma}

\begin{proof}
Suppose, by way of contradiction, that there exist two elements of $Y$, $y_1$ and $y_2$, whose
difference is a square mod $m^{2k}$.
By the definition of $Y$, we can write those elements as

\begin{itemize}
\item
$y_1=m^2x_1 + s_1$, where $x_1\in X$ and $s_1\in B$.
\item
$y_2=m^2x_2 + s_2$, where $x_2\in X$ and $s_2\in B$.
\end{itemize}

Since $s_1,s_2\in B$, 

\begin{itemize}
\item
$s_1=mz_1+b_1$, where $z_1\in \{0,\ldots,m-1\}$ and $b_1\in S$.
\item
$s_2=mz_2+b_2$, where $z_2\in \{0,\ldots,m-1\}$ and $b_2\in S$.
\end{itemize}

Hence

\begin{itemize}
\item
$y_1=m^2x_1 + mz_1+b_1$, where $x_1\in X$, $z_1\in \{0,\ldots,m-1\}$, and $b_1\in S$.
\item
$y_2=m^2x_2 + mz_2+b_2$, where $x_2\in X$, $z_2\in \{0,\ldots,m-1\}$, and $b_2\in S$.
\end{itemize}

Since $y_1-y_2$ is a square mod $m^{2k}$ there exists $a,L$ such that

$$y_1-y_2=a^2+L_1m^{2k}$$

$$m^2(x_1-x_2) + m(z_1-z_2) + (b_1 - b_2) = a^2 + Lm^{2k}.$$

Reducing this equation mod $m$, we obtain

$$b_1 - b_2 \equiv  a^2 \pmod m.$$

By the definition of $S$, $b_1=b_2$, so we have

$$a^2 \equiv 0 \pmod m.$$

Since $m$ divides $a^2$, and $m$ is squarefree,  $m$ divides $a$.
Hence $a=cm$, so $a^2=c^2m^2$. Thus we have

$$m^2(x_1-x_2) + m(z_1-z_2) = c^2m^2 + Lm^{2k}.$$

Reducing this equation mod $m^2$, and using the fact that $k\ge 1$, we obtain

$$m(z_1-z_2) \equiv  0 \pmod {m^2}.$$

Since $0\le z_1,z_2 \le m-1$, we have $m\,|z_1-z_2| < m^2$, hence
$z_1=z_2$.

Since $b_1=b_2$ and $z_1=z_2$, we now have
$$m^2(x_1-x_2) = c^2m^2 + Lm^{2k}$$

Dividing by $m^2$, we obtain
$$(x_1-x_2) = c^2 + Lm^{2k-2}.$$

Recall that $x_1,x_2\in X$. By the condition on $X$, there do not
exist two elements of $X$ whose difference is a square mod $m^{2k-2}$.
Since the last equation states that the difference of two elements of
$X$ is a square mod $m^{2k-2}$, this is a contradiction.
\end{proof}

\begin{lemma}\label{le:iterate}
For all $k\ge 1$, 
$\smod {m^{2k}} \ge m\cdot\smod{m}\cdot\smod {m^{2k-2}}.$
\end{lemma}

\begin{proof}
Let $S$ be an $\sdfmod{m}$ set of size $\smod{m}$
and  let $X$ be an $\sdfmod{m^{2k-2}}$ set of size
$\smod{m^{2k-2}}$.
By Lemma~\ref{le:mk} there exists $Y$, an $\sdfmod{m^{2k}}$ set
of size $m\cdot\smod{m}\cdot\smod {m^{2k-2}}$.
Hence
$$\smod {m^{2k}} \ge m\cdot\smod{m}\cdot\smod {m^{2k-2}}.$$
\end{proof}

\begin{lemma}\label{le:build}
Assume that there exists a squarefree $m$ and a set $S\subseteq [m]$ such
that $S$ is $\sdfmod m$.
Then $\s n \ge \Omega(n^{\half(1+\log_m |S| )})$.
(The constant implicit in the $\Omega$ depends on $m$.)
\end{lemma}

\begin{proof}
By the premise, $\smod m \ge |S|.$
By Lemma~\ref{le:iterate}

$$(\forall k\ge 1)[\smod {m^{2k}} \ge m|S|\smod{m^{2k-2}}].$$

Hence

$$\smod {m^{2k}} \ge (m|S|)^{k}\smod{1}=(m|S|)^k$$

Let $n=m^{2k}$, so $k=\log_m \sqrt n$. Then

\[
\begin{array}{rl}
(m|S|)^{k}=&m^k(|S|)^k\cr
		=&m^k(|S|)^{\log_m \sqrt n}\cr
            =&n^{0.5} |S|^{\log_m \sqrt n}\cr
\end{array}
\]

Note that

$$|S|^{\log_m \sqrt n}= (\sqrt n)^{\log_m |S|} = n^{0.5\log_m |S|}.$$

Hence, for $n=m^{2k}$,

$$
\smod{m^{2k}}\ge(m|S|)^k=n^{0.5} n^{0.5\log_m |S|}=n^{\half(1+\log_m |S|)}.
$$

Thus we have

$$\smod{n} \ge \Omega(n^{\half(1+\log_m |S|)}).$$

\end{proof}

\begin{theorem}\label{th:main}~
$\s n \ge \Omega(n^{\log_{205} 12}) \ge  \Omega(\newsize)$.
\end{theorem}

\begin{proof}
Let $m=205$ and 
$S = \{0, 2, 8, 14, 77, 79, 85, 96, 103, 109, 111, 181\}$.
Clearly, $m$ is square free. An easy calculation shows that there are no two elements of $S$ whose difference is a square mod $m$.
Note that $\log_m |S| =\log_{205} 12 >0.4668$.
Hence, by Lemma~\ref{le:build},

$$
\s n\ge \Omega(n^{\half(1+\log_m |S|)}) \ge \Omega(n^{\half(1+0.4668)})\ge \Omega(\newsize).
$$

\end{proof}

\begin{note}
Ruzsa used $m=65$ and $S=7$ to obtain his results.
\end{note}

\section{Square-Difference-Free Colorings}

The following is an easy corollary of Theorem~\ref{th:polyvdw}.

\begin{theorem}\label{th:qvdw}
There exists a function $f:\nat\into\nat$ 
such that the following holds:

If $[f(c)]$ is $c$-colored, there will be $x<y$
such that $x$ and $y$ are the same color and $y-x$ is a square.
\end{theorem}

The proof of Theorem~\ref{th:polyvdw}~\cite{pvdww}
gives an enormous upper bound on $f$.
We can use our results to provide a lower bound.
The idea is as follows: Take a square-difference-free set, and
translate it $c$ times to obtain a $c$-coloring.

\begin{definition}
Let $A\subseteq [n]$.  $B$ is a {\it translate of $A$ relative to $n$}
if there exists $t$ such that
$$B= \{ x+t \st x\in A \} \cap [n].$$
We will omit the ``relative to $n$'' when $n$ is clear from context.
\end{definition}

The following lemma is Lemma 4.4 from~\cite{multiparty}; however,
it can also be viewed as a special case of the Symmetric Hypergraph
Lemma~\cite{GRS}.

\begin{lemma}\label{le:tran}
Let $A\subseteq [n]$.
There exist $c\le O(\frac{n\log n}{|A|})$ and
sets $A_1,\ldots,A_c$ that are translates of $A$ such that
$[n]=A_1 \cup \cdots \cup A_c$.
(Note that the lemma holds for any set $A$; however,
we will apply it when $A$ is \sdfns.)
\end{lemma}

\begin{lemma}\label{le:coloring}
Let $c,n$ be such that 
$c\le O(\frac{n\log n}{\s n})$, and
let $f$ be the function from Theorem~\ref{th:qvdw}.
Then $f(c)\ge n$.
\end{lemma}

\begin{proof}
Let $A\subseteq [n]$ be an \sdf set of size $\s n$.
By Lemma~\ref{le:tran},
there exist $c\le O(\frac{n\log n}{\s n})$ translates of $A$
such that the union of the translates covers all of $[n]$.
Call the translates $A_1,\ldots,A_c$.
Let $\COL$ be the $c$-coloring of $[n]$ that maps
a number $x$ to the least $i$ such that $x\in A_i$.
For $1\le i\le c$ let $C_i$ be the set of numbers that are colored $i$.
Since each $C_i$ is an \sdfns, this coloring has no
$x,y\in [n]$, $x\ne y$,  such that $\COL(x)=\COL(y)$ and $|x-y|$ is a square.
Hence $f(c)\ge n$.
\end{proof}

\begin{theorem}
$f(c) \ge \Omega(\cbtwo).$
\end{theorem}

\begin{proof}
Fix $c$. We want to find an $n$ as small as possible such that

$$c\le O\biggl (\frac{n\log n}{\s n}\biggr).$$

Since $\s n \ge \Omega(\newsize)$ 

$$\biggl (\frac{n\log n}{\newsize}\biggr) \le O(\onemsizep).$$

Hence it will suffice to find an $n$ as small as possible such that

$$c\le O(\onemsizep).$$

We can take

$$n \ge \Omega(\cbone) \ge \Omega(\cbtwo).$$

Hence $f(c) \ge \Omega(\cbtwo).$

\end{proof}

\section{Open Problem}

Combining the upper bound of \cite{PSS} with our lower bound we have:

$$
\Omega(\newsize)\le\s n\le O\bigg(\frac{n}{\log^{c_n} n} \bigg )
$$
where $c_n\goes\infty$.

The open problem is to close this gap.
One way to raise the lower bounds is to find values of $m$ and $|S|$ that
satisfy the premise of Lemma~\ref{le:build} with a larger value
of $\log_m |S|$ then we obtained.

\section{Acknowledgments}

We would like to thank Georgia Martin and Edward Gan for
proofreading and commentary.
We would like to thank
Boris Bukh, Ben Green, and Nikos Frantzikinakis
for pointing us to the papers~\cite{ruzsasq}
and ~\cite{PSS}.

\end{document}